\documentclass[11pt,oneside,leqno]{amsart}
\allowdisplaybreaks
\usepackage[margin=2.5cm]{geometry}
\usepackage{multirow,tabularx}
\newcolumntype{Y}{>{\centering\arraybackslash}X}
\usepackage{pbox}
\usepackage{parskip}
\usepackage{eucal}
\usepackage{amscd}
\usepackage{amssymb}
\usepackage{amsmath}
\usepackage{amsfonts}
\usepackage{amsopn}
\usepackage{latexsym}
\usepackage{tikz}
\usepackage{comment}
\usepackage{ulem}
\usepackage{hyperref}
\usepackage[nodisplayskipstretch]{setspace}
\usepackage{caption}
\usepackage[OT2,T1]{fontenc}
\DeclareSymbolFont{cyrletters}{OT2}{wncyr}{m}{n}
\DeclareMathSymbol{\Sha}{\mathalpha}{cyrletters}{"58}
\newcommand{\genlegendre}[4]{%
  \genfrac{(}{)}{}{#1}{#3}{#4}%
  \if\relax\detokenize{#2}\relax\else_{\!#2}\fi
}

\usepackage{enumerate}
\usepackage[utf8]{inputenc}
\makeatletter
\@namedef{subjclassname@2010}{%
  \textup{2010} Mathematics Subject Classification}
\makeatother

\newtheorem{thm}{Theorem}[section]

\newtheorem{lem}[thm]{Lemma}

\makeatletter
\newcommand{\BIG}{\bBigg@{2}}
\newcommand{\vast}{\bBigg@{3}}
\newcommand{\Vast}{\bBigg@{5}}
\makeatother

\numberwithin{equation}{section}

\begin{document}
\setlength{\arrayrulewidth}{0.1mm}
%\setlength{\tabcolsep}{18pt}
%\baselineskip{=17pt}

%%%%%%%%%%%

%% In the running head, replace first names by initials
%% and give an abbreviation of the title.

\title[Heron triangle and Selmer group]{Heronian elliptic curves and the size of the $2$-Selmer group}

\author[Debopam Chakraborty]{Debopam Chakraborty}
\address{Department of Mathematics\\ BITS-Pilani, Hyderabad campus\\
Hyderabad, INDIA}
\email{debopam@hyderabad.bits-pilani.ac.in}

\author[Vinodkumar Ghale]{Vinodkumar Ghale}
\address{Department of Mathematics\\ BITS-Pilani, Hyderabad campus\\
Hyderabad, INDIA}
\email{p20180465@hyderabad.bits-pilani.ac.in}

\date{}


\maketitle

\section*{Abstract}
\noindent A generalization of the congruent number problem is to find positive integers $n$ that appear as the areas of Heron triangles. Selmer group of a congruent number elliptic curve has been studied quite extensively. Here, we look into the $2$-Selmer group structure for Heronian elliptic curves associated with Heron triangles of area $n$ and one of the angle $\theta$ such that $\tan \frac{\theta}{2} = n^{-1}$ and $n^{2}+1=2q$ for some prime $q$.

\section{Introduction}
\noindent A square-free positive integer $n$ is called a \textit{congruent number} if it appears as an area of a right-angle triangle with rational sides. The problem of identifying such positive integers $n$ is known as \textit{the congruent number problem}. The congruent number problem boils down to the rank calculation of the elliptic curve $E: y^{2} = x^{3} - n^{2}x$, known as \textit{the congruent number elliptic curve}. Selmer group of an elliptic curve holds crucial information about the rank of the elliptic curve (cf. \cite{Silverman} for a detailed definition of the Selmer group along with full $n$-descent method for the computation of the $n$-Selmer group of an elliptic curve). For the congruent number elliptic curve, the size of the Selmer group has been discussed in detail in the works of various authors, most notably in a series of two papers by Heath-Brown (cf. \cite{Brown1} and \cite{Brown2}). A simple but elegant way to compute the $2$-Selmer rank follows from the appendix of \cite{Brown2}, through the works of Monsky, which turns the problem of finding the $2$-Selmer rank in a problem of linear algebra over $\mathbb{F}_{2}$.\\
\noindent A \textit{Heron triangle} is a triangle with rational sides without the constraint of being a right-angle triangle. An immediate generalization of the congruent number problem is to identify a positive square-free integer $n$ as the area of a Heron triangle. The work of Goins and Maddox (cf. \cite{Goins}) gave rise to an elliptic curve, known as \textit{the Heronian elliptic curve} associated with a Heron triangle. In its general form the curve is $E: y^{2} = x(x-n\tau)(x+n {\tau}^{-1})$ where $\tau = \tan \frac{\theta}{2}$ is one of the angles $\theta$ of the Heron triangle under consideration. Not much is known yet about the Selmer group of a Heronian elliptic curve. In a recent work (cf. \cite{Chakraborty}), Chakraborty et al. gave a family of Heronian elliptic curves with $2$-Selmer rank precisely $1$, associated with the Heron triangle of area $2^{m}p$ and $\tau = \tan \frac{\theta}{2} = 2^{m}$ for some integer $m$ and prime $p$. A lack of Monsky-type matrices (cf. Appendix of \cite{Brown2}) that is available for the congruent number elliptic curves makes the problem of finding the $2$-Selmer rank mostly dependent on a ``full $2$-descent'' here. \\
\noindent The main results of this article give an explicit group structure for the $2$-Selmer group of $E_{n}: y^{2} = x(x-1)(x+n^{2})$ associated with a Heron triangle of area $n$ and one of the angles $\theta$ such that $\tau = \tan \frac{\theta}{2} = n^{-1}$ such that $n^{2}+1 = 2q$ for some prime $q$. A rank computation for a similar curve with prime $n$ was carried out in \cite{Ghale}. The $2$-Selmer group of an elliptic curve $E$ is denoted by $S^{(2)}(E/ \mathbb{Q})$. The $2$-Selmer rank of $E$, denoted by $s^{(2)}(E/ \mathbb{Q})$, is defined as $|S^{(2)}(E/ \mathbb{Q})| = 2^{{k}+s^{(2)}(E/ \mathbb{Q})}$ where $k=2$ or $4$ for a Heronian elliptic curve $E$ (cf. \cite{Goins}). For a positive integer $n$, we denote the number of $k \pmod 8$ prime factors of $n$ by $\Omega_{k,n}$. We now state the results below.
\begin{thm}\label{thm1}
For a square-free odd integer $n$ such that $n^{2}+1 = 2q$ for some prime $q$, let $E_{n}: y^{2} = x(x-1)(x+n^{2})$ denotes the Heronian elliptic curve associated with the Heron triangle of area $n$ and one of the angles $\theta$ such that $\tan(\frac{\theta}{2}) = n^{-1}$. Then, \\  
$[a]$ $s^{(2)}(E_{n}/ \mathbb{Q}) \cong (\mathbb{Z}/ 2 \mathbb{Z})^{\Omega_{1,n}+1}$ if $\Omega_{5,n} = 0$.\\
$[b]$ $s^{(2)}(E_{n}/ \mathbb{Q}) \cong (\mathbb{Z}/ 2 \mathbb{Z})^{\left(\Omega_{1,n} + \frac{\Omega_{5,n}(\Omega_{5,n}-1)}{2}\right)}$ if $\Omega_{5,n}  \neq 0$.
\end{thm}
\noindent For an even integer $n$, we state the following result below. We assume $n^{2}+1 = q$ for a prime $q$ now. 
\begin{thm}\label{thm2}
    For a square-free even integer $n$ such that $n^{2}+1 = q$ for some prime $q$, let $E_{n}: y^{2} = x(x-1)(x+n^{2})$ is the Heronian elliptic curve as above. Then $s^{(2)}(E_{n}/ \mathbb{Q}) \cong (\mathbb{Z}/ 2 \mathbb{Z})^{\left (\Omega_{1,n} + \Omega_{5,n}\right)}$.
\end{thm}
\noindent We note that the above results can be used as a tool to compute the Heronian elliptic curve with arbitrarily large $2$-Selmer rank. This is because the $2$-Selmer rank here is directly related to the number of prime factors of $n$ of the form $1, 5$ modulo $8$. 
\section{Bounding the $2$-Selmer rank} 
\noindent In this section, we look into the $2$-Selmer group $S^{(2)}(E_{n}/ \mathbb{Q})$. We note that under the $2$-descent method, each element $(b_{1},b_{2})$ in $S^{(2)}(E_{n}/ \mathbb{Q})$ corresponds to the following homogeneous space with local solutions everywhere. 
\begin{align}
& b_1z_1^2 - b_2z_2^2 = 1, \label{eq22}\\
& b_1z_1^2 - b_1b_2z_3^2 = -n^{2}, \label{eq23}
\end{align}
such that $(z_1,z_2,z_3) \in \mathbb{Q}^* \times \mathbb{Q}^* \times \mathbb{Q}$, $b_{1}, b_{2}$ are square-free integers whose only prime factors are $2, q$ and prime factors of $n$. A detailed version of the $2$-descent method is available in \cite{Silverman}. We note that the image of $E_{n}(\mathbb{Q})_{\text{tors}}$ under the $2$-descent map contains $\{(1,1), (1,2q), (-1,-1), (-1, -2q)\}$ when $n$ is odd and  $\{(1,1), (1,q), (-1,-1), (-1, -q)\}$ when $n$ is even (cf. \cite{Chakraborty} for a similar calculation). We start with the following result regarding $l$-adic solutions of the homogeneous space given by equations (\ref{eq22}) and (\ref{eq23}). The proof is quite straightforward but is included for the sake of the completeness of the paper.\\
\begin{lem}\label{lem1}
Let $(z_{1},z_{2},z_{3})$ be a solution to the homogeneous space given by equations (\ref{eq22}) and (\ref{eq23}). Then for each prime $l < \infty$ and all $i \in \{1,2,3\}$, either $v_{l}(z_{i}) \geq 0$ or $v_{l}(z_{i}) = -k$ for some positive integer $k$.
\end{lem}
\begin{proof}
For $i \in \{1,2\}$, $v_{l}(z_{i}) < 0$ immediately implies $v_{l}(z_{i}) = -k$ for some positive integer $k$ (cf. Lemma $3.1$ in \cite{Chakraborty} for a similar and detailed proof). If $v_{l}(z_{3}) < 0$ and $v_{l}(z_{1}) > v_{l}(z_{3})$, then either from (\ref{eq23}) or from subtracting (\ref{eq23}) from (\ref{eq22}), we get $l^{3}$ divides $b_{1}b_{2}$, an absurdity. Hence $v_{l}(z_{1}) \leq v_{l}(z_{3}) < 0$, and using the first case of the proof, we conclude.
\end{proof}
\noindent Noting that we will only look into possible $(b_{1},b_{2}) \in S^{(2)}(E_{n}/ \mathbb{Q})/ Im(E_{n}(\mathbb{Q})_{\text{tors}})$ under the $2$-descent map, without loss of generality, we fix the following assumptions. \\
\noindent \textbf{Assumption 1:} $b_{1} > 0, b_{2} > 0$ always (due to $l = \infty$).\\
\noindent \textbf{Assumption 2:} $b_{2}$ is odd if $n$ is odd and $b_{2} \not \equiv 0 \pmod q$ if $n$ is even. 
\begin{lem}\label{lem2}
    Let $n$ is odd and $(b_{1},b_{2}) \in S^{(2)}(E_{n}/ \mathbb{Q})$. Then for an arbitrary prime $p$, \\
    $[a]$ $b_{1} \equiv 0 \pmod p \implies p \equiv 1,5 \pmod 8$.\\
    $[b]$ $b_{2} \in \{1,q\}$. Moreover, $n \equiv 0 \pmod p$, $p \equiv 5 \pmod 8$ $\implies$ $(1,q) \not \in S^{(2)}(E_{n}/ \mathbb{Q})$.\\
    $[c]$ The number of prime factors of $b_{1}$ of the form $5$ modulo $8$ must always be even.
\end{lem}
\begin{proof}
    We start with noticing that $b_{1} \not \equiv 0 \pmod q$. Otherwise, from (\ref{eq23}) and (\ref{eq22}), $v_{q}(z_{i}) \geq 0 \implies -n^{2} \equiv 0 \pmod q$ and, $v_{q}(z_{i}) = -k \implies b_{1} \equiv 0 \pmod {q^{2}}$, contradiction in each case. This implies no $q$-adic solution for the homogeneous space associated with $(b_{1},b_{2})$ when $q$ divides $b_{1}$.\\
    A very similar argument as above shows that $b_{1}$ is odd if $(b_{1},b_{2}) \in S^{(2)}(E_{n}/ \mathbb{Q})$. As otherwise, $v_{2}(z_{i}) \geq 0$ implies $2$ divides $-n^{2}$ from (\ref{eq23}) and $v_{2}(z_{i}) = -k$ leads to $b_{2}$ as even from (\ref{eq22}), both contradictions from our assumptions of $n$ and $b_{2}$.\\
    We now note that $b_{2} \not \equiv 0 \pmod p$ if $p$ divides $n$. If $v_{p}(z_{i}) \geq 0$ then subtracting (\ref{eq23}) from (\ref{eq22}), one can get $2q \equiv 0 \pmod p$, a contradiction. Otherwise, from (\ref{eq23}), we get that $p^{2}$ divides $b_{1}$, a contradiction. This, along with the assumptions above, narrows down the choices of $b_{2} \in \{1,q\}$.\\
    Now $b_{1} \equiv 0 \pmod p$ implies $p$ divides $n$. We note that possible elements of $S^{(2)}(E_{n}/ \mathbb{Q})$ now looks like $(b_{1},1)$ or $(b_{1},q)$.\\
    If $p \equiv 3,7 \pmod 8$, we notice that the homogeneous space corresponding to $(b_{1},1)$ has no $p$-adic solution. This is because if $p \equiv 3,7 \pmod 8$ then depending on $v_{p}(z_{i})$, from (\ref{eq22}), either $\big(\frac{-1}{p}\big) = 1$, or $b_{2} \equiv 0 \pmod p$, contradiction either way.\\
    Now for homogeneous spaces corresponding to $(b_{1},q)$, we first note that $v_{p}(z_{i}) = - k \implies q \equiv 0 \pmod p$, a contradiction. For $v_{p}(z_{i}) \geq 0$, subtracting (\ref{eq23}) from (\ref{eq22}), we get that $\big(\frac{-2}{p}\big) = 1$ which implies $p \not \equiv 7 \pmod 8$. For $p \equiv 3 \pmod 8$, we note that $\big(\frac{q}{p}\big) = -1$ which contradicts (\ref{eq22}). This concludes the proof of part $[a]$. \\
    To conclude the proof of part $[b]$, we show that whenever $n$ has a prime factor $p \equiv 5 \pmod 8$, the homogeneous space for $(1,q)$ has no $p$-adic solution. From (\ref{eq23}), we get $\big(\frac{q}{p}\big) = 1$, a contradiction as $\big(\frac{q}{p}\big) = -1$ here. \\
    For part $[c]$, we notice that the only possible element in the $2$-Selmer group is of the form $(b_{1},1)$ with $p \equiv 1,5 \pmod 8$ its only prime divisors where $p$ varies over all divisors of $n$. For $p \equiv 5 \pmod 8$, we note that $\big(\frac{p}{q}\big) = -1$ whereas $\big(\frac{p}{q}\big) = 1$ whenever $p \equiv 1 \pmod 8$. Noting that only factors of $b_{1}$ are now primes $p \equiv 1,5 \pmod 8$, one can get that $\big(\frac{b_{1}}{q}\big) = 1$ if there are even number of factors of $5 \pmod 8$, else $\big(\frac{b_{1}}{q}\big) = -1$. The result now follows from equations (\ref{eq22}) and (\ref{eq23}), observing the fact that $\big(\frac{b_{1}}{q}\big) = 1$ is a necessary condition for the homogeneous space corresponding to $(b_{1},1)$ to have a $q$-adic solution.
\end{proof}
\begin{lem}\label{lem3}
    Let $n$ is even and $(b_{1},b_{2}) \in S^{(2)}(E_{n}/ \mathbb{Q})$. Then $b_{2} = 1$ always and for an arbitrary prime $p$, $b_{1} \equiv 0 \pmod p \implies p \equiv 1,5 \pmod 8$ and $p \neq q$.
\end{lem}
\begin{proof}
    We first note that $b_{1} \not \equiv 0 \pmod q$ here, and the proof is the same as in the $n$ odd case. Similarly, we can also see that if $p$ is an odd prime that divides $n$, then $b_{2} \not \equiv 0 \pmod p$.\\
    We now notice that $b_{2}$ is odd if $(b_{1},b_{2}) \in S^{(2)}(E_{n}/ \mathbb{Q})$. One can trivially observe that $\gcd(b_{1}, b_{2}) \not \equiv 0 \pmod 2$. Now $b_{2}$ is even either implies $2$ divides $q$ when $v_{2}(z_{i}) \geq 0$ or $2$ divides $b_{1}$ when $v_{2}(z_{i}) = -k$, contradiction either way.\\
    We have now established that $b_{1} \not \equiv 0 \pmod q$ and $b_{2} = 1$ if $(b_{1},b_{2}) \in S^{(2)}(E_{n}/ \mathbb{Q})$.\\
    From equation (\ref{eq22}), it is evident that $b_{1} \equiv 0 \pmod p \implies p \equiv 1,5 \pmod 8$. This is because (\ref{eq22}) implies $\big(\frac{-1}{p}\big) = 1$ for the pair $(b_{1},1)$. $n^{2}+1 = q$ implies $q \equiv 5 \pmod 8$. Subtracting (\ref{eq23}) from (\ref{eq22}) for homogeneous space corresponding to $(2,1)$, we get $\big(\frac{2}{q}\big) = 1$, a contradiction for $q \equiv 5 \pmod 8$. Hence the homogeneous space corresponding to $(2,1)$ has no $q$-adic solution. This concludes the proof. 
\end{proof}
\section{Everywhere Local Solution}
\noindent We prove that the homogeneous spaces corresponding to $(p,1)$ and $(1,q)$ have local solutions everywhere where $p$ is any prime factor of $n$ such that $p \equiv 1 \pmod 8$ ($p \equiv 1,5 \pmod 8$ if $n$ is even). We use Hensel's lemma to lift a simple root of a polynomial $f(x)$ modulo a prime $l$ to a solution for $f(x)$ in $\mathbb{Z}_{l}$. Let $C$ be the homogeneous space given by (\ref{eq22}) and (\ref{eq23}) corresponding to the pairs $(p,1)$ and $(1,q)$. A trivial application of smoothness of $C$, the degree-genus formula, and the Hasse-Weil bound give the following.\\
\noindent $[a]$ For $l\geq 5, \text{ } l \neq t$ where $t \equiv 1,5 \pmod 8$ and $t$ divides $n$, $C$ is a homogeneous space of genus $1$ corresponding to $(p,1)$ or $(p_{1}p_{2}, 1)$ with $\# C(\mathbb{F}_{l})\;\geq \;1+l-2\sqrt{l}\;\geq \;2$ where $p\equiv 1 \pmod 8$ for odd $n$ ($p \equiv 1,5 \pmod 8$ for even $n$), $p_{1} \equiv p_{2} \equiv 5 \pmod 8$ for odd $n$. \\
\noindent $[b]$ For $l\geq 5, \text{ } l \neq t$ where $t \equiv 1,3 \pmod 8$ and $t$ divides $n$, $C$ is a homogeneous space of genus $1$ corresponding to $(1,q)$ then $\# C(\mathbb{F}_{l})\;\geq \;1+l-2\sqrt{l}\;\geq \;2$.\\
A simple application of Hensel's lemma now immediately (cf. \cite{Chakraborty} for a similar and detailed proof) implies that the homogeneous spaces mentioned above have $l$-adic solution for all the primes $l$ mentioned above. This reduces the problem to finding local solutions for only finitely many primes.
\begin{lem}\label{lem4}
    Let $n$ be an odd integer such that $n^{2}+1 = 2q$ for a prime $q$. Then for each prime factor, $p$ of $n$, $p \equiv 1 \pmod 8$, the homogeneous spaces corresponding to $(p,1)$ have local solutions everywhere for $l \leq \infty$. 
\end{lem}
\begin{proof}
    As mentioned above, we need only to show local solutions exist for $l=2,3$ and $t$ where $t \equiv 1,5 \pmod 8$ is a prime that divides $n$. Fixing two of the three variables $z_{1},z_{2},z_{3}$, we present a set of simple roots for the system of equations (\ref{eq22}) and (\ref{eq23}) modulo $l$ using Lemma \ref{lem1} that can be lifted to $\mathbb{Q}_{l}$ using Hensel's lemma.\\
    For $l=2$, $z_{1} = 1$ is a simple root modulo $8$ to the system of equations $pz_{1}^{2} - 1 = 2^{2k}$ and $z_{1}^{2} - 1 = - \frac{n^{2}}{p} \cdot 2^{2k}$ for $k \geq 2$.\\
    For $l=3$, $z_{1} = 1$ is a simple root modulo $3$ to the system of equations $pz_{1}^{2} - 1 = 3^{2k}$ and $z_{1}^{2} - 1 = - \frac{n^{2}}{p} \cdot 3^{2k}$ when $p \equiv 1 \pmod 3$. When $p \equiv 2 \pmod 3$, one can see that $z_{1} = 1$ is a simple root modulo $3$ to the simultaneous equations $pz_{1}^{2} - 1 = 1$ and $z_{1}^{2} - 0 = - \frac{n^{2}}{p}$.\\
    For $l = t$, $t \equiv 1, 5 \pmod 8$ and $t$ divides $n$, $z_{2} = a$ such that $a^{2} \equiv -1 \pmod p$ is a simple root modulo $p$ for equations $p \cdot 0^{2} -z_{2}^{2} = 1$ and $p \cdot 0^{2} - z_{2}^{2} = 2q$. This concludes the proof. 
\end{proof}
\noindent In a very similar way to that of Lemma \ref{lem4}, we can prove the following result for $(p_{1}p_{2},1)$ where $p_{i} \equiv 5 \pmod 8$. We only observe that $p_{1}p_{2} \equiv 1 \pmod 8$ in this case. 
\begin{lem}\label{lem7}
    Let $n$ be an odd integer such that $n^{2}+1 = 2q$ for a prime $q$. Then if exist, for any two prime factors $p_{1}$ and $p_{2}$ of $n$, $p_{i} \equiv 5 \pmod 8, i= 1,2$, the homogeneous spaces corresponding to $(p_{1}p_{2},1)$ have local solutions everywhere for $l \leq \infty$.
\end{lem}
\begin{lem}\label{lem5}
    Let $n$ be an even integer such that $n^{2}+1 = q$ for a prime $q$. Then for each prime factor, $p$ of $n$, $p \equiv 1,5 \pmod 8$, the homogeneous spaces corresponding to $(p,1)$ have local solutions everywhere for $l \leq \infty$.
\end{lem}
\begin{proof}
    We note that the proof is exactly the same as the proof given in Lemma \ref{lem4} for $p \equiv 1,5 \pmod 8$, except the $l=2$ case for $p \equiv 5 \pmod 8$, which we describe below. \\
    For $p \equiv 5 \pmod 8$, we note that $z_{1} = 1$ is a simple root modulo $8$ to the system of equations $pz_{1}^{2} - 1 = 2^{2}$ and $z_{1}^{2} - 1 = - \frac{n^{2}}{p} \cdot 2^{2}$. This concludes the proof.\\
\end{proof}
\begin{lem}\label{lem6}
    For $n$ odd, $(1,q) \in S^{(2)}(E_{n}/ \mathbb{Q})$ if $n$ has no prime factor of the form $5$ modulo $8$.
\end{lem}
\begin{proof}
    We have already established in Lemma \ref{lem2} that if $n$ has a prime divisor of the form $5$ modulo $8$, then $(1,q) \not \in S^{(2)}(E_{n}/ \mathbb{Q})$.\\
    For $l=2$, we notice $n \not \equiv 1 \pmod 8$ implies $q \equiv 5 \pmod 8$. For homogeneous space associated to $(1,q)$, this implies $z_{1} = 1$ is a simple root to $z_{1}^{2} - q \cdot 1^{2} = 2^{2}$ and $z_{1}^{2} - q \cdot 1^{2} = -n^{2} \cdot 2^{2}$. Now $n \equiv 1 \pmod 8 \implies q \equiv 1 \pmod 8$. We can immediately now see that $z_{1} = 1$ is a simple root modulo $8$ of the simultaneous equations $z_{1}^{2} - q = 2^{2k}$ and $z_{1}^{2} - q = - n^{2} \cdot 2^{2k}$ for $k \geq 2$. \\
    For $l = 3$, we first note that $q \equiv 1 \pmod 3$ always. Then $z_{1} = 1$ is a simple root to the system of equations $z_{1}^{2} - q \cdot 1^{2} = 3^{2k}$ and $z_{1}^{2} - q \cdot 1^{2} = -n^{2} \cdot 3^{2k}$.\\
    For $l = t$, where $t \equiv 1,3 \pmod 8$ is prime, $t$ divides $n$, we note that $\big(\frac{-q}{t}\big) = \big(\frac{-2}{t}\big) = 1$. Now we can see that $z_{2} = a$ is a simple root to the equations $0^{2} - qz_{2}^{2} = 1$ and $0^{2} - z_{2}^{2} = 2 $ where $a^{2} \equiv -2 \pmod t$. This concludes the proof. 
\end{proof}
\section{Size of the $2$-Selmer group}
\noindent We are now in a position to bind the size of the $2$-Selmer group $S^{(2)}(E_{n}/ \mathbb{Q})$. The results are now almost evident from the workings established in earlier sections.\\ 
\begin{proof}[Proof of Theorem \ref{thm1}]
From Lemma \ref{lem6}, we know that $(1,q) \in S^{(2)}(E_{n}/ \mathbb{Q})$ if $n$ has no prime factor of the form $5$ modulo $8$. From the group structure of the $2$-Selmer group, Lemma \ref{lem2} and Lemma \ref{lem4}, we can see that $S^{(2)}(E_{n}/ \mathbb{Q}) = \langle{(p_{i},1), (1,q)\rangle}$ where $p_{i} \equiv 1 \pmod 8$ vary over all the prime factors of $n$. this proves part $[a]$ of Theorem \ref{thm1}.\\
From Lemma \ref{lem2}, we know $(1,q) \not \in S^{(2)}(E_{n}/ \mathbb{Q})$ if $n$ has a prime factor of the form $5$ modulo $8$. Hence Lemma \ref{lem4} and Lemma \ref{lem7} imply $S^{(2)}(E_{n}/ \mathbb{Q}) = \langle{(p_{i},1), (t_{i}t_{j},1)\rangle}$ where $p_{i}$ varies over all the prime factors of $n$ of the form $1$ modulo $8$ and $t_{i} \neq t_{j}$ varies over all the prime factors of $n$ of the form $5$ modulo $8$. Because there are $\frac{\Omega_{5,n}(\Omega_{5,n}-1)}{2}$ ways to choose distinct $t_{i}, t_{j}$, the result follows. 
\end{proof}

\begin{proof}[Proof of Theorem \ref{thm2}] We observe that  from Lemma \ref{lem3}, $(b_{1},b_{2}) \in S^{(2)}(E_{n}/ \mathbb{Q})$ implies $(b_{1},b_{2}) = (p,1)$ such that $p \equiv 1,5 \pmod 8$ divides $n$ are the only possibilities. In Lemma \ref{lem5}, we established the homogeneous spaces corresponding to those pairs have local solutions everywhere; hence the result follows.     
\end{proof}
\noindent We list both odd and even integers $n$ below with the 2-Selmer ranks of corresponding Heronian elliptic curve $E_{n}$. The table is verified using MAGMA{\cite{Magma}} and SAGE {\cite{Sage}}.   
\begin{table}[htbp]
\caption{ Examples}
\centering

\begin{tabular}{ |c|c|c|c|}
\hline
%\multicolumn{2}{c}{From MAgma and Sage}\multicolumn{2}{c}{From Theorem 1.1,1.2}
  $n$ & $q$ &  2- Selmer Rank of $E_{n}$ & Generators \\
  \hline
$3\cdot 5$  & 113 & 0 & - \\
 \hline
 $3\cdot 5 \cdot 7 \cdot 11$ & 667013 & 0 & - \\
 \hline
$5\cdot 11 \cdot  13$ & 255613 & 1 & (65,1)  \\
\hline
$5\cdot 17$ & 3613 & 1 & (17,1) \\
 \hline
$17 \cdot  23$  & 76441 & 2 & (17,1),(1, 76441)\\
 \hline
 $7 \cdot  11 \cdot 13$ & 501001 & 0 & - \\
 \hline
 $3 \cdot  5 \cdot 7\cdot 11 \cdot 19$  & 240791513 & 0 & -\\
 \hline
 $5 \cdot  13 \cdot 3$  & 19013 & 1 & (65,1) \\
 \hline
 $2 \cdot 3 \cdot 11$ & 4357 & 0 & - \\
 \hline
  $2 \cdot  5 \cdot 13$ & 16901 & 2 & (5,1),(13,1) \\
 \hline
  $2 \cdot  7 \cdot 29$ & 164837 & 1 & (29,1)\\
 \hline
  $2 \cdot  5 \cdot 17$ & 28901 & 2 & (5,1),(17,1)\\
  \hline
   $2 \cdot  7 \cdot 17 \cdot 23$ & 29964677 & 1 & (17,1) \\
 \hline
  $2 \cdot  3 \cdot5 \cdot 7 \cdot 11$ & 5336101 & 1 & (5,1) \\
 \hline
  $ 2 \cdot  3 \cdot 5 \cdot 7 \cdot 13$ & 7452901 & 2 & (5,1),(13,1)\\
 \hline
  $ 2 \cdot  3 \cdot 5 \cdot 7 \cdot 37$ & 7452901 & 2 & (5,1),(37,1) \\
 \hline
 $ 2 \cdot  3 \cdot 5 \cdot 7 \cdot 41$ & 74132101 & 2 & (5,1), (41,1)\\
 \hline
  $ 17 \cdot  73$  & 770041 & 3 & (17,1),(73,1),(1,77041) \\
  \hline
 $ 17 \cdot  41\cdot   97 $ & 2285488441 & 4 & (17,1),(41,1),(97,1),(1, 2285488441)\\
 \hline
 $ 2 \cdot  5\cdot   13 \cdot 17 $ & 4884101 & 3 & (5,1),(13,1),(17,1)\\
 \hline
\end{tabular}

\end{table}

\end{document}